\documentclass{amsart}
\usepackage{amssymb}
\usepackage{amsfonts}

\newtheorem{theorem}{Theorem}
\theoremstyle{plain}

\newtheorem{corollary}{Corollary}

\newtheorem{definition}{Definition}
\newtheorem{example}{Example}

\newtheorem{lemma}{Lemma}

\newtheorem{proposition}{Proposition}

\numberwithin{equation}{section}
\input{tcilatex}

\begin{document}
\title[Hadamard's type inequalities]{On new approach Hadamard-type
inequalities for $s$-geometrically convex functions}
\author{MEVL\"{U}T TUN\c{C}$^{\square }$}
\address{$^{\square }$Department of Mathematics, Faculty of Science and
Arts, Kilis 7 Aral\i k University, Kilis, 79000, Turkey.}
\email{$^{\square }$mevluttunc@kilis.edu.tr}
\author{\.{I}BRAH\.{I}M KARABAYIR$^{\triangle }$}
\address{$^{\triangle }$The Institute for Graduate Studies in Sciences and
Engineering, Kilis 7 Aral\i k University, Kilis, 79000, Turkey.}
\email{$^{\triangle }$ikarabayir@kilis.edu.tr}
\date{November 15, 2012}
\thanks{This paper is in final form and no version of it will be submitted
for publication elsewhere.\\
$^{\square }$Corresponding Author}

\begin{abstract}
In this paper we achieve some new Hadamard type inequalities using
elementary well known inequalities for functions whose first derivatives
absolute values are $s$-geometrically and geometrically convex. And also we
get some applications for special means for positive numbers.
\end{abstract}

\maketitle

\section{INTRODUCTION}

Let \ $f:I\subseteq 
%TCIMACRO{\U{211d} }%
%BeginExpansion
\mathbb{R}
%EndExpansion
\rightarrow 
%TCIMACRO{\U{211d} }%
%BeginExpansion
\mathbb{R}
%EndExpansion
$ be a convex mapping defined on the interval $I$\ of real numbers and $%
a,b\in I,$ with $a<b.$ The following double inequalities:%
\begin{equation*}
f\left( \frac{a+b}{2}\right) \leq \int_{a}^{b}f\left( x\right) dx\leq \frac{%
f\left( a\right) +f\left( b\right) }{2}
\end{equation*}%
hold. This double inequality is known in the literature as the
Hermite-Hadamard inequality for convex functions.

In recent years many authors established several inequalities connected to
this fact. For recent results, refinements, counterparts, generalizations
and new Hermite-Hadamard-type inequalities see \cite{hud}-\cite{pec}.

In this section we will present definitions and some results used in this
paper.

\begin{definition}
Let $I$ be an interval in $%
%TCIMACRO{\U{211d} }%
%BeginExpansion
\mathbb{R}
%EndExpansion
.$ Then $f:I\rightarrow 
%TCIMACRO{\U{211d} }%
%BeginExpansion
\mathbb{R}
%EndExpansion
,\emptyset \neq I\subseteq 
%TCIMACRO{\U{211d} }%
%BeginExpansion
\mathbb{R}
%EndExpansion
$\ is said to be convex if 
\begin{equation}
f\left( tx+\left( 1-t\right) y\right) \leq tf\left( x\right) +\left(
1-t\right) f\left( y\right) .  \label{7}
\end{equation}%
for all $x,y\in I$ and $t\in \left[ 0,1\right] .$
\end{definition}

\begin{definition}
\cite{hud}\textit{\ Let }$s\in \left( 0,1\right] .$\textit{\ A function }$%
f:I\subset 
%TCIMACRO{\U{211d} }%
%BeginExpansion
\mathbb{R}
%EndExpansion
_{0}=\left[ 0,\infty \right) \rightarrow 
%TCIMACRO{\U{211d} }%
%BeginExpansion
\mathbb{R}
%EndExpansion
_{0}$\textit{\ is said to be }$s$-\textit{convex in the second sense if \ \
\ \ \ \ \ \ \ \ \ \ }%
\begin{equation}
f\left( tx+\left( 1-t\right) y\right) \leq t^{s}f\left( x\right) +\left(
1-t\right) ^{s}f\left( y\right)  \label{2}
\end{equation}%
\textit{for all }$x,y\in I$\textit{\ and }$t\in \left[ 0,1\right] $\textit{.}
\end{definition}

It can be easily checked for $s=1$, $s$-convexity reduces to the ordinary
convexity of functions defined on $\left[ 0,\infty \right) $.

Recently, in \cite{zh}, the concept of geometrically and $s$-geometrically
convex functions was introduced as follows.

\begin{definition}
\cite{zh} A function $f:I\subset 
%TCIMACRO{\U{211d} }%
%BeginExpansion
\mathbb{R}
%EndExpansion
_{+}=\left( 0,\infty \right) \rightarrow 
%TCIMACRO{\U{211d} }%
%BeginExpansion
\mathbb{R}
%EndExpansion
_{+}$\textit{\ }is said to be a geometrically convex function if%
\begin{equation}
f\left( x^{t}y^{1-t}\right) \leq \left[ f\left( x\right) \right] ^{t}\left[
f\left( y\right) \right] ^{1-t}  \label{3}
\end{equation}%
for \textit{all }$x,y\in I$\textit{\ and }$t\in \left[ 0,1\right] $\textit{.}
\end{definition}

\begin{definition}
\cite{zh} A function $f:I\subset 
%TCIMACRO{\U{211d} }%
%BeginExpansion
\mathbb{R}
%EndExpansion
_{+}\rightarrow 
%TCIMACRO{\U{211d} }%
%BeginExpansion
\mathbb{R}
%EndExpansion
_{+}$\textit{\ }is said to be a $s$-geometrically convex function if%
\begin{equation}
f\left( x^{t}y^{1-t}\right) \leq \left[ f\left( x\right) \right] ^{t^{s}}%
\left[ f\left( y\right) \right] ^{\left( 1-t\right) ^{s}}  \label{4}
\end{equation}%
for some $s\in \left( 0,1\right] $, where $x,y\in I$\textit{\ and }$t\in %
\left[ 0,1\right] $\textit{.}
\end{definition}

If $s=1$, the $s$-geometrically convex function becomes a geometrically
convex function on $%
%TCIMACRO{\U{211d} }%
%BeginExpansion
\mathbb{R}
%EndExpansion
_{+}$.

\begin{example}
\cite{zh} Let $f\left( x\right) =x^{s}/s,$ $x\in \left( 0,1\right] , $ $%
0<s<1,$ $q\geq 1,$ and then the function%
\begin{equation}
\left\vert f^{\prime }\left( x\right) \right\vert ^{q}=x^{\left( s-1\right)
q}  \label{5}
\end{equation}%
is monotonically decreasing on $\left( 0,1\right] $. For $t\in \left[ 0,1%
\right] $, we have%
\begin{equation}
\left( s-1\right) q\left( t^{s}-t\right) \leq 0,\text{ \ \ }\left(
s-1\right) q\left( \left( 1-t\right) ^{s}-\left( 1-t\right) \right) \leq 0.
\label{6}
\end{equation}%
Hence, $\left\vert f^{\prime }\left( x\right) \right\vert ^{q}$ is $s$%
-geometrically convex on $\left( 0,1\right] $ for $0<s<1$.
\end{example}

\section{Hadamard's type inequalities}

In order to prove our main theorems, we need the following lemma \cite{lmmm}.

\begin{lemma}
\label{l1}\cite{lmmm} Let $f:\ I\subset 
%TCIMACRO{\U{211d} }%
%BeginExpansion
\mathbb{R}
%EndExpansion
\rightarrow 
%TCIMACRO{\U{211d} }%
%BeginExpansion
\mathbb{R}
%EndExpansion
$ be a differentiable mapping on $I^{\circ }$\textit{\ where }$a,b\in $ $I$
with $a$ $<$ $b.$If $f^{\prime }$ $\in L\left[ a,b\right] ,$ then the
following equality holds: 
\begin{eqnarray}
&&\frac{f\left( a\right) +f\left( b\right) }{2}-\frac{1}{b-a}%
\int_{a}^{b}f\left( x\right) dx  \label{a} \\
&=&\frac{b-a}{4}\left[ \int_{0}^{1}\left( -t\right) f^{\prime }\left( \frac{%
1+t}{2}a+\frac{1-t}{2}b\right) dt+\int_{0}^{1}tf^{\prime }\left( \frac{1+t}{2%
}b+\frac{1-t}{2}a\right) dt\right] .  \notag
\end{eqnarray}
\end{lemma}

A simple proof of this equality can be also done integrating by parts in the
right hand side. The details are left to the interested reader.

The next theorems gives a new result of the upper Hermite-Hadamard
inequality for $s$-geometrically convex functions.

In the following part of the paper;%
\begin{equation}
\alpha \left( u,v\right) =\left\vert f^{\prime }\left( a\right) \right\vert
^{u}\left\vert f^{\prime }\left( b\right) \right\vert ^{-v},\text{ \ }%
u,v\geq 0,  \label{111}
\end{equation}%
\begin{equation}
g_{1}\left( \alpha \right) =\left\{ 
\begin{array}{cc}
\frac{1}{2} & \alpha =1 \\ 
\frac{\alpha \ln \alpha -\alpha +1}{\left( \ln \alpha \right) ^{2}} & \alpha
\neq 1%
\end{array}%
\right.  \label{112}
\end{equation}%
and%
\begin{equation}
g_{2}\left( \alpha \right) =\left\{ 
\begin{array}{cc}
1 & \alpha =1 \\ 
\frac{\alpha -1}{\ln \alpha } & \alpha \neq 1%
\end{array}%
\right.  \label{113}
\end{equation}

\begin{theorem}
\label{t1}Let $f:I\subset 
%TCIMACRO{\U{211d} }%
%BeginExpansion
\mathbb{R}
%EndExpansion
_{+}\rightarrow 
%TCIMACRO{\U{211d} }%
%BeginExpansion
\mathbb{R}
%EndExpansion
_{+}$ be differentiable mapping on $I^{\circ },$ $a,b\in I$ with $a<b$ and $%
f^{\prime }$ is integrable on $\left[ a,b\right] .$ If $\ \left\vert
f^{\prime }\right\vert $ is $s$-geometrically convex and monotonically
decreasing on $\left[ a,b\right] ,$ and $s\in \left( 0,1\right] $ then the
following inequality holds:%
\begin{eqnarray}
&&\left\vert \frac{f\left( a\right) +f\left( b\right) }{2}-\frac{1}{b-a}%
\int_{a}^{b}f\left( x\right) dx\right\vert  \label{b} \\
&\leq &\frac{\left( b-a\right) }{4}\left\vert f^{\prime }\left( a\right)
f^{\prime }\left( b\right) \right\vert ^{\frac{s}{2}}\left( g_{1}\left(
\alpha \left( \frac{s}{2},\frac{s}{2}\right) \right) +g_{1}\left( \alpha
\left( \frac{-s}{2},\frac{-s}{2}\right) \right) \right) .  \notag
\end{eqnarray}
\end{theorem}

\begin{proof}
Since $\left\vert f^{\prime }\right\vert $ is a $s$-geometrically convex and
monotonically decreasing on $\left[ a,b\right] ,$ from Lemma \ref{l1}, we get%
\begin{eqnarray*}
&&\left\vert \frac{f\left( a\right) +f\left( b\right) }{2}-\frac{1}{b-a}%
\int_{a}^{b}f\left( x\right) dx\right\vert \\
&\leq &\frac{b-a}{4}\left[ \int_{0}^{1}\left\vert -t\right\vert \left\vert
f^{\prime }\left( \frac{1+t}{2}a+\frac{1-t}{2}b\right) \right\vert
dt+\int_{0}^{1}\left\vert t\right\vert \left\vert f^{\prime }\left( \frac{1+t%
}{2}b+\frac{1-t}{2}a\right) \right\vert dt\right] \\
&\leq &\frac{b-a}{4}\left[ \int_{0}^{1}\left\vert -t\right\vert \left\vert
f^{\prime }\left( a^{\frac{1+t}{2}}b^{\frac{1-t}{2}}\right) \right\vert
dt+\int_{0}^{1}\left\vert t\right\vert \left\vert f^{\prime }\left( b^{\frac{%
1+t}{2}}a^{\frac{1-t}{2}}\right) \right\vert dt\right] \\
&\leq &\frac{b-a}{4}\left[ \int_{0}^{1}\left\vert -t\right\vert \left\vert
f^{\prime }\left( a\right) \right\vert ^{\left( \frac{1+t}{2}\right)
^{s}}\left\vert f^{\prime }\left( b\right) \right\vert ^{\left( \frac{1-t}{2}%
\right) ^{s}}dt+\int_{0}^{1}\left\vert t\right\vert \left\vert f^{\prime
}\left( b\right) \right\vert ^{\left( \frac{1+t}{2}\right) ^{s}}\left\vert
f^{\prime }\left( a\right) \right\vert ^{\left( \frac{1-t}{2}\right) ^{s}}dt%
\right]
\end{eqnarray*}%
If $\ 0<k\leq 1,$ $0<m,n\leq 1,\ \ \ $%
\begin{equation}
k^{m^{n}}\leq k^{mn}  \label{1}
\end{equation}%
When $\left\vert f^{\prime }\left( a\right) \right\vert =\left\vert
f^{\prime }\left( b\right) \right\vert =1,$ by( \ref{1}), we get%
\begin{equation*}
\left\vert \frac{f\left( a\right) +f\left( b\right) }{2}-\frac{1}{b-a}%
\int_{a}^{b}f\left( x\right) dx\right\vert \leq \frac{b-a}{4}
\end{equation*}%
When $0<\left\vert f^{\prime }\left( a\right) \right\vert ,\left\vert
f^{\prime }\left( b\right) \right\vert <1,$ by( \ref{1}), we get 
\begin{eqnarray*}
&&\left\vert \frac{f\left( a\right) +f\left( b\right) }{2}-\frac{1}{b-a}%
\int_{a}^{b}f\left( x\right) dx\right\vert \\
&\leq &\frac{b-a}{4}\left[ \int_{0}^{1}\left\vert -t\right\vert \left\vert
f^{\prime }\left( a\right) \right\vert ^{s\left( \frac{1+t}{2}\right)
}\left\vert f^{\prime }\left( b\right) \right\vert ^{s\left( \frac{1-t}{2}%
\right) }dt+\int_{0}^{1}\left\vert t\right\vert \left\vert f^{\prime }\left(
b\right) \right\vert ^{s\left( \frac{1+t}{2}\right) }\left\vert f^{\prime
}\left( a\right) \right\vert ^{s\left( \frac{1-t}{2}\right) }dt\right] \\
&=&\frac{b-a}{4}\left\vert f^{\prime }\left( a\right) f^{\prime }\left(
b\right) \right\vert ^{\frac{s}{2}}\left[ \int_{0}^{1}\left\vert
-t\right\vert \left\vert \frac{f^{\prime }\left( a\right) }{f^{\prime
}\left( b\right) }\right\vert ^{\frac{st}{2}}dt+\int_{0}^{1}\left\vert
t\right\vert \left\vert \frac{f^{\prime }\left( b\right) }{f^{\prime }\left(
a\right) }\right\vert ^{\frac{st}{2}}dt\right] \\
&=&\frac{\left( b-a\right) }{4}\left\vert f^{\prime }\left( a\right)
f^{\prime }\left( b\right) \right\vert ^{\frac{s}{2}}\left( g_{1}\left(
\alpha \left( \frac{s}{2},\frac{s}{2}\right) \right) +g_{1}\left( \alpha
\left( \frac{-s}{2},\frac{-s}{2}\right) \right) \right)
\end{eqnarray*}%
which completes the proof.
\end{proof}

\begin{theorem}
\label{t2}Let $f:I\subset 
%TCIMACRO{\U{211d} }%
%BeginExpansion
\mathbb{R}
%EndExpansion
_{+}\rightarrow 
%TCIMACRO{\U{211d} }%
%BeginExpansion
\mathbb{R}
%EndExpansion
_{+}$ be differentiable on $I^{\circ },$ $a,b\in I$, with $a<b$ and $%
f^{\prime }\in L\left( \left[ a,b\right] \right) .$ If $\left\vert f^{\prime
}\right\vert ^{q}$ is $s$-geometrically convex and monotonically decreasing
on $\left[ a,b\right] $ for $p,q>1$ and $s\in \left( 0,1\right] ,$ then%
\begin{eqnarray}
&&  \label{g} \\
&&\left\vert \frac{f\left( a\right) +f\left( b\right) }{2}-\frac{1}{b-a}%
\int_{a}^{b}f\left( x\right) dx\right\vert  \notag \\
&\leq &\frac{\left( b-a\right) }{4\left( p+1\right) ^{\frac{1}{p}}}%
\left\vert f^{\prime }\left( a\right) f^{\prime }\left( b\right) \right\vert
^{\frac{s}{2}}\left\{ \left[ g_{2}\left( \alpha \left( \frac{sq}{2},\frac{sq%
}{2}\right) \right) \right] ^{\frac{1}{q}}+\left[ g_{2}\left( \alpha \left( 
\frac{-sq}{2},\frac{-sq}{2}\right) \right) \right] ^{\frac{1}{q}}\right\} , 
\notag
\end{eqnarray}%
where $\frac{1}{p}+\frac{1}{q}=1.$
\end{theorem}

\begin{proof}
Since $\left\vert f^{\prime }\right\vert ^{q}$ is a $s$-geometrically convex
and monotonically decreasing on $\left[ a,b\right] $, from Lemma \ref{l1}
and the well known H\"{o}lder inequality, we have%
\begin{eqnarray*}
&&\left\vert \frac{f\left( a\right) +f\left( b\right) }{2}-\frac{1}{b-a}%
\int_{a}^{b}f\left( x\right) dx\right\vert \\
&\leq &\frac{b-a}{4}\left[ \int_{0}^{1}\left\vert -t\right\vert \left\vert
f^{\prime }\left( \frac{1+t}{2}a+\frac{1-t}{2}b\right) \right\vert
dt+\int_{0}^{1}\left\vert t\right\vert \left\vert f^{\prime }\left( \frac{1+t%
}{2}b+\frac{1-t}{2}a\right) \right\vert dt\right] \\
&\leq &\frac{b-a}{4}\left\{ \left[ \int_{0}^{1}t^{p}dt\right] ^{\frac{1}{p}}%
\left[ \int_{0}^{1}\left\vert f^{\prime }\left( a^{\frac{1+t}{2}}b^{\frac{1-t%
}{2}}\right) \right\vert ^{q}dt\right] ^{\frac{1}{q}}\right. \\
&&+\left. \left[ \int_{0}^{1}t^{p}dt\right] ^{\frac{1}{p}}\left[
\int_{0}^{1}\left\vert f^{\prime }\left( b^{\frac{1+t}{2}}a^{\frac{1-t}{2}%
}\right) \right\vert ^{q}dt\right] ^{\frac{1}{q}}\right\} \\
&=&\frac{b-a}{4\left( p+1\right) ^{\frac{1}{p}}}\left\{ \left[
\int_{0}^{1}\left\vert f^{\prime }\left( a^{\frac{1+t}{2}}b^{\frac{1-t}{2}%
}\right) \right\vert ^{q}dt\right] ^{\frac{1}{q}}+\left[ \int_{0}^{1}\left%
\vert f^{\prime }\left( b^{\frac{1+t}{2}}a^{\frac{1-t}{2}}\right)
\right\vert ^{q}dt\right] ^{\frac{1}{q}}\right\} \\
&\leq &\frac{b-a}{4\left( p+1\right) ^{\frac{1}{p}}}\left\{ \left[
\int_{0}^{1}\left( \left\vert f^{\prime }\left( a\right) \right\vert
^{\left( \frac{1+t}{2}\right) ^{s}}\left\vert f^{\prime }\left( b\right)
\right\vert ^{\left( \frac{1-t}{2}\right) ^{s}}\right) ^{q}dt\right] ^{\frac{%
1}{q}}\right. \\
&&+\left. \left[ \int_{0}^{1}\left( \left\vert f^{\prime }\left( b\right)
\right\vert ^{\left( \frac{1+t}{2}\right) ^{s}}\left\vert f^{\prime }\left(
a\right) \right\vert ^{\left( \frac{1-t}{2}\right) ^{s}}\right) ^{q}dt\right]
^{\frac{1}{q}}\right\}
\end{eqnarray*}%
When $\left\vert f^{\prime }\left( a\right) \right\vert =\left\vert
f^{\prime }\left( b\right) \right\vert =1,$ by( \ref{1}), we get%
\begin{equation*}
\left\vert \frac{f\left( a\right) +f\left( b\right) }{2}-\frac{1}{b-a}%
\int_{a}^{b}f\left( x\right) dx\right\vert \leq \frac{b-a}{2\left(
p+1\right) ^{\frac{1}{p}}}
\end{equation*}%
When $0<\left\vert f^{\prime }\left( a\right) \right\vert ,\left\vert
f^{\prime }\left( b\right) \right\vert <1,$ by( \ref{1}), we get

\begin{eqnarray*}
&\leq &\frac{b-a}{4\left( p+1\right) ^{\frac{1}{p}}}\left\{ \left[
\int_{0}^{1}\left\vert f^{\prime }\left( a\right) \right\vert ^{sq\left( 
\frac{1+t}{2}\right) }\left\vert f^{\prime }\left( b\right) \right\vert
^{sq\left( \frac{1-t}{2}\right) }dt\right] ^{\frac{1}{q}}\right. \\
&&+\left. \left[ \int_{0}^{1}\left\vert f^{\prime }\left( b\right)
\right\vert ^{sq\left( \frac{1+t}{2}\right) }\left\vert f^{\prime }\left(
a\right) \right\vert ^{sq\left( \frac{1-t}{2}\right) }dt\right] ^{\frac{1}{q}%
}\right\} \\
&=&\frac{\left( b-a\right) }{4\left( p+1\right) ^{\frac{1}{p}}}\left\vert
f^{\prime }\left( a\right) f^{\prime }\left( b\right) \right\vert ^{\frac{s}{%
2}}\left\{ \left( \int_{0}^{1}\left\vert \frac{f^{\prime }\left( a\right) }{%
f^{\prime }\left( b\right) }\right\vert ^{\frac{sq}{2}t}dt\right) ^{\frac{1}{%
q}}+\left( \int_{0}^{1}\left\vert \frac{f^{\prime }\left( b\right) }{%
f^{\prime }\left( a\right) }\right\vert ^{\frac{sq}{2}t}dt\right) ^{\frac{1}{%
q}}\right\} \\
&=&\frac{\left( b-a\right) }{4\left( p+1\right) ^{\frac{1}{p}}}\left\vert
f^{\prime }\left( a\right) f^{\prime }\left( b\right) \right\vert ^{\frac{s}{%
2}}\left\{ \left[ g_{2}\left( \alpha \left( \frac{sq}{2},\frac{sq}{2}\right)
\right) \right] ^{\frac{1}{q}}+\left[ g_{2}\left( \alpha \left( \frac{-sq}{2}%
,\frac{-sq}{2}\right) \right) \right] ^{\frac{1}{q}}\right\}
\end{eqnarray*}%
which completes the proof.
\end{proof}

\begin{corollary}
Let $f:I\subseteq \left( 0,\infty \right) \rightarrow \left( 0,\infty
\right) $ be differentiable on $I^{\circ },$ $a,b\in I$ with $a<b,$ and $%
f^{\prime }\in L\left( \left[ a,b\right] \right) .$ If $\left\vert f^{\prime
}\right\vert ^{q}$ is $s$-geometrically convex and monotonically decreasing
on $\left[ a,b\right] $ for $p,q>1$ and $s\in \left( 0,1\right] ,$ then

i) When $p=q=2$, one has%
\begin{eqnarray*}
&&\left\vert \frac{f\left( a\right) +f\left( b\right) }{2}-\frac{1}{b-a}%
\int_{a}^{b}f\left( x\right) dx\right\vert \\
&\leq &\frac{\left( b-a\right) }{4\sqrt{3}}\left\vert f^{\prime }\left(
a\right) f^{\prime }\left( b\right) \right\vert ^{\frac{s}{2}}\left\{ \sqrt{%
g_{2}\left( \alpha \left( s,s\right) \right) }+\sqrt{g_{2}\left( \alpha
\left( -s,-s\right) \right) }\right\}
\end{eqnarray*}

ii) If we take $s=1$ in (\ref{g}), we have for geometrically convex, one has%
\begin{eqnarray*}
&&\left\vert \frac{f\left( a\right) +f\left( b\right) }{2}-\frac{1}{b-a}%
\int_{a}^{b}f\left( x\right) dx\right\vert \\
&\leq &\frac{\left( b-a\right) }{4\left( p+1\right) ^{\frac{1}{p}}}%
\left\vert f^{\prime }\left( a\right) f^{\prime }\left( b\right) \right\vert
^{\frac{1}{2}}\left\{ \left[ g_{2}\left( \alpha \left( q,q\right) \right) %
\right] ^{\frac{1}{q}}+\left[ g_{2}\left( \alpha \left( -q,-q\right) \right) %
\right] ^{\frac{1}{q}}\right\}
\end{eqnarray*}
\end{corollary}

\begin{theorem}
\label{t3}Let $f:I\subset 
%TCIMACRO{\U{211d} }%
%BeginExpansion
\mathbb{R}
%EndExpansion
_{+}\rightarrow 
%TCIMACRO{\U{211d} }%
%BeginExpansion
\mathbb{R}
%EndExpansion
_{+}$ be differentiable on $I^{\circ },$ $a,b\in I$ \ with $a<b$ and $%
f^{\prime }\in L(\left[ a\text{,}b\right] )$. If $\left\vert f^{\prime
}\right\vert ^{q}$ is $s$-geometrically convex and monotonically decreasing
on $\left[ a\text{,}b\right] ,$ for $q\geq 1$ and $s\in \left( 0,1\right] ,$
then%
\begin{eqnarray}
&&  \label{xx} \\
&&\left\vert \frac{f\left( a\right) +f\left( b\right) }{2}-\frac{1}{b-a}%
\int_{a}^{b}f\left( x\right) dx\right\vert  \notag \\
&\leq &\frac{b-a}{4}\left( \frac{1}{2}\right) ^{1-\frac{1}{q}}\left\{
\left\vert \frac{f^{\prime }\left( a\right) }{f^{\prime }\left( b\right) }%
\right\vert ^{\frac{s}{2}}\left[ g_{1}\left( \alpha \left( \frac{sq}{2},%
\frac{sq}{2}\right) \right) \right] ^{\frac{1}{q}}+\left\vert \frac{%
f^{\prime }\left( b\right) }{f^{\prime }\left( a\right) }\right\vert ^{\frac{%
s}{2}}\left[ g_{1}\left( \alpha \left( \frac{-sq}{2},\frac{-sq}{2}\right)
\right) \right] ^{\frac{1}{q}}\right\} .  \notag
\end{eqnarray}
\end{theorem}

\begin{proof}
Since $\left\vert f^{\prime }\right\vert ^{q}$ is a $s$-geometrically convex
and monotonically decreasing on $\left[ a,b\right] $, from Lemma \ref{l1}
and the well known power mean integral inequality, we have 
\begin{eqnarray*}
&&\left\vert \frac{f\left( a\right) +f\left( b\right) }{2}-\frac{1}{b-a}%
\int_{a}^{b}f\left( x\right) dx\right\vert \\
&\leq &\frac{b-a}{4}\left[ \int_{0}^{1}\left\vert -t\right\vert \left\vert
f^{\prime }\left( \frac{1+t}{2}a+\frac{1-t}{2}b\right) \right\vert
dt+\int_{0}^{1}\left\vert t\right\vert \left\vert f^{\prime }\left( \frac{1+t%
}{2}b+\frac{1-t}{2}a\right) \right\vert dt\right] \\
&\leq &\frac{b-a}{4}\left\{ \left[ \int_{0}^{1}tdt\right] ^{1-\frac{1}{q}}%
\left[ \int_{0}^{1}t\left\vert f^{\prime }\left( \frac{1+t}{2}a+\frac{1-t}{2}%
b\right) \right\vert ^{q}dt\right] ^{\frac{1}{q}}\right. \\
&&\left. +\left[ \int_{0}^{1}tdt\right] ^{1-\frac{1}{q}}\left[
\int_{0}^{1}t\left\vert f^{\prime }\left( \frac{1+t}{2}b+\frac{1-t}{2}%
a\right) \right\vert ^{q}dt\right] ^{\frac{1}{q}}\right\} \\
&\leq &\frac{b-a}{4}\left( \frac{1}{2}\right) ^{1-\frac{1}{q}}\left\{ \left[
\int_{0}^{1}t\left\vert f^{\prime }\left( a^{\frac{1+t}{2}}b^{\frac{1-t}{2}%
}\right) \right\vert ^{q}dt\right] ^{\frac{1}{q}}+\left[ \int_{0}^{1}\left%
\vert f^{\prime }\left( b^{\frac{1+t}{2}}a^{\frac{1-t}{2}}\right)
\right\vert ^{q}dt\right] ^{\frac{1}{q}}\right\} \\
&\leq &\frac{b-a}{4}\left( \frac{1}{2}\right) ^{1-\frac{1}{q}}\left\{ \left[
\int_{0}^{1}t\left\vert f^{\prime }\left( a\right) ^{q\left( \frac{1+t}{2}%
\right) ^{s}}f^{\prime }\left( b\right) ^{q\left( \frac{1-t}{2}\right)
^{s}}\right\vert dt\right] ^{\frac{1}{q}}\right. \\
&&+\left. \left[ \int_{0}^{1}t\left\vert f^{\prime }\left( b\right)
^{q\left( \frac{1+t}{2}\right) ^{s}}f^{\prime }\left( a\right) ^{q\left( 
\frac{1-t}{2}\right) ^{s}}\right\vert dt\right] ^{\frac{1}{q}}\right\}
\end{eqnarray*}%
When $\left\vert f^{\prime }\left( a\right) \right\vert =\left\vert
f^{\prime }\left( b\right) \right\vert =1,$ by( \ref{1}), we get%
\begin{equation*}
\left\vert \frac{f\left( a\right) +f\left( b\right) }{2}-\frac{1}{b-a}%
\int_{a}^{b}f\left( x\right) dx\right\vert \leq \frac{b-a}{4}
\end{equation*}%
When $0<\left\vert f^{\prime }\left( a\right) \right\vert ,\left\vert
f^{\prime }\left( b\right) \right\vert <1,$ by( \ref{1}), we get%
\begin{eqnarray*}
&&\left\vert \frac{f\left( a\right) +f\left( b\right) }{2}-\frac{1}{b-a}%
\int_{a}^{b}f\left( x\right) dx\right\vert \\
&\leq &\frac{b-a}{4}\left( \frac{1}{2}\right) ^{1-\frac{1}{q}}\left\{ \left[
\int_{0}^{1}t\left\vert f^{\prime }\left( a\right) ^{sq\left( \frac{1+t}{2}%
\right) }f^{\prime }\left( b\right) ^{sq\left( \frac{1-t}{2}\right)
}\right\vert dt\right] ^{\frac{1}{q}}\right. \\
&&+\left. \left[ \int_{0}^{1}t\left\vert f^{\prime }\left( b\right)
^{sq\left( \frac{1+t}{2}\right) }f^{\prime }\left( a\right) ^{sq\left( \frac{%
1-t}{2}\right) }\right\vert dt\right] ^{\frac{1}{q}}\right\} \\
&\leq &\frac{b-a}{4}\left( \frac{1}{2}\right) ^{1-\frac{1}{q}}\left\{ \left[
\left\vert \frac{f^{\prime }\left( a\right) }{f^{\prime }\left( b\right) }%
\right\vert ^{\frac{sq}{2}}\int_{0}^{1}t\left\vert \frac{f^{\prime }\left(
a\right) }{f^{\prime }\left( b\right) }\right\vert ^{\frac{sq}{2}t}dt\right]
^{\frac{1}{q}}+\left[ \left\vert \frac{f^{\prime }\left( b\right) }{%
f^{\prime }\left( a\right) }\right\vert ^{\frac{sq}{2}}\int_{0}^{1}t\left%
\vert \frac{f^{\prime }\left( b\right) }{f^{\prime }\left( a\right) }%
\right\vert ^{\frac{sq}{2}t}dt\right] ^{\frac{1}{q}}\right\} \\
&=&\frac{b-a}{4}\left( \frac{1}{2}\right) ^{1-\frac{1}{q}}\left\{ \left\vert 
\frac{f^{\prime }\left( a\right) }{f^{\prime }\left( b\right) }\right\vert ^{%
\frac{s}{2}}\left[ \int_{0}^{1}t\left\vert \frac{f^{\prime }\left( a\right) 
}{f^{\prime }\left( b\right) }\right\vert ^{\frac{sq}{2}t}dt\right] ^{\frac{1%
}{q}}+\left\vert \frac{f^{\prime }\left( b\right) }{f^{\prime }\left(
a\right) }\right\vert ^{\frac{s}{2}}\left[ \int_{0}^{1}t\left\vert \frac{%
f^{\prime }\left( b\right) }{f^{\prime }\left( a\right) }\right\vert ^{\frac{%
sq}{2}t}dt\right] ^{\frac{1}{q}}\right\} \\
&=&\frac{b-a}{4}\left( \frac{1}{2}\right) ^{1-\frac{1}{q}}\left\{ \left\vert 
\frac{f^{\prime }\left( a\right) }{f^{\prime }\left( b\right) }\right\vert ^{%
\frac{s}{2}}\left[ g_{1}\left( \alpha \left( \frac{sq}{2},\frac{sq}{2}%
\right) \right) \right] ^{\frac{1}{q}}+\left\vert \frac{f^{\prime }\left(
b\right) }{f^{\prime }\left( a\right) }\right\vert ^{\frac{s}{2}}\left[
g_{1}\left( \alpha \left( \frac{-sq}{2},\frac{-sq}{2}\right) \right) \right]
^{\frac{1}{q}}\right\}
\end{eqnarray*}

\textit{which completes the proof.}
\end{proof}

\begin{theorem}
\label{t4}Let $f:I\subset 
%TCIMACRO{\U{211d} }%
%BeginExpansion
\mathbb{R}
%EndExpansion
_{+}\rightarrow 
%TCIMACRO{\U{211d} }%
%BeginExpansion
\mathbb{R}
%EndExpansion
_{+}$ be differentiable on $I^{\circ },$ $a,b\in I$, with $a<b$ and $%
f^{\prime }\in L\left( \left[ a,b\right] \right) .$ If $\left\vert f^{\prime
}\right\vert $ is $s$-geometrically convex and monotonically decreasing on $%
\left[ a,b\right] $ for $\mu _{1},\mu _{2},\eta _{1},\eta _{2}>0$ with $\mu
_{1}+\eta _{1}=1$ and $\mu _{2}+\eta _{2}=1$ and $s\in \left( 0,1\right] ,$
then%
\begin{eqnarray}
&&\left\vert \frac{f\left( a\right) +f\left( b\right) }{2}-\frac{1}{b-a}%
\int_{a}^{b}f\left( x\right) dx\right\vert  \label{123} \\
&\leq &\frac{b-a}{4}\left\vert f^{\prime }\left( a\right) f^{\prime }\left(
b\right) \right\vert ^{\frac{s}{2}}\left\{ \frac{\left( 1+\mu _{2}\right)
\mu _{1}^{2}+\left( 1+\mu _{1}\right) \mu _{2}^{2}}{\left( 1+\mu _{1}\right)
\left( 1+\mu _{2}\right) }\right.  \notag \\
&&+\left. \eta _{1}g_{2}\left( \alpha \left( \frac{s}{2\eta _{1}},\frac{s}{%
2\eta _{1}}\right) \right) +\eta _{2}g_{2}\left( \alpha \left( \frac{s}{%
2\eta _{2}},\frac{s}{2\eta _{2}}\right) \right) \right\} .  \notag
\end{eqnarray}
\end{theorem}

\begin{proof}
Since $\left\vert f^{\prime }\right\vert $ is a $s$-geometrically convex and
monotonically decreasing on $\left[ a,b\right] $, from Lemma \ref{l1}, we
have%
\begin{eqnarray*}
&&\left\vert \frac{f\left( a\right) +f\left( b\right) }{2}-\frac{1}{b-a}%
\int_{a}^{b}f\left( x\right) dx\right\vert \\
&\leq &\frac{b-a}{4}\left[ \int_{0}^{1}\left\vert -t\right\vert \left\vert
f^{\prime }\left( \frac{1+t}{2}a+\frac{1-t}{2}b\right) \right\vert
dt+\int_{0}^{1}\left\vert t\right\vert \left\vert f^{\prime }\left( \frac{1+t%
}{2}b+\frac{1-t}{2}a\right) \right\vert dt\right] \\
&\leq &\frac{b-a}{4}\left[ \int_{0}^{1}\left\vert -t\right\vert \left\vert
f^{\prime }\left( a^{\frac{1+t}{2}}b^{\frac{1-t}{2}}\right) \right\vert
dt+\int_{0}^{1}\left\vert t\right\vert \left\vert f^{\prime }\left( b^{\frac{%
1+t}{2}}a^{\frac{1-t}{2}}\right) \right\vert dt\right] \\
&\leq &\frac{b-a}{4}\left[ \int_{0}^{1}\left\vert -t\right\vert \left\vert
f^{\prime }\left( a\right) \right\vert ^{\left( \frac{1+t}{2}\right)
^{s}}\left\vert f^{\prime }\left( b\right) \right\vert ^{\left( \frac{1-t}{2}%
\right) ^{s}}dt+\int_{0}^{1}\left\vert t\right\vert \left\vert f^{\prime
}\left( b\right) \right\vert ^{\left( \frac{1+t}{2}\right) ^{s}}\left\vert
f^{\prime }\left( a\right) \right\vert ^{\left( \frac{1-t}{2}\right) ^{s}}dt%
\right]
\end{eqnarray*}%
When $0<\left\vert f^{\prime }\left( a\right) \right\vert ,\left\vert
f^{\prime }\left( b\right) \right\vert \leq 1,$ by( \ref{1}), we get 
\begin{eqnarray}
&&  \label{k} \\
&&\left\vert \frac{f\left( a\right) +f\left( b\right) }{2}-\frac{1}{b-a}%
\int_{a}^{b}f\left( x\right) dx\right\vert  \notag \\
&\leq &\frac{b-a}{4}\left[ \int_{0}^{1}\left\vert -t\right\vert \left\vert
f^{\prime }\left( a\right) \right\vert ^{s\left( \frac{1+t}{2}\right)
}\left\vert f^{\prime }\left( b\right) \right\vert ^{s\left( \frac{1-t}{2}%
\right) }dt+\int_{0}^{1}\left\vert t\right\vert \left\vert f^{\prime }\left(
b\right) \right\vert ^{s\left( \frac{1+t}{2}\right) }\left\vert f^{\prime
}\left( a\right) \right\vert ^{s\left( \frac{1-t}{2}\right) }dt\right] 
\notag \\
&=&\frac{b-a}{4}\left\vert f^{\prime }\left( a\right) f^{\prime }\left(
b\right) \right\vert ^{\frac{s}{2}}\left[ \int_{0}^{1}\left\vert
-t\right\vert \left\vert \frac{f^{\prime }\left( a\right) }{f^{\prime
}\left( b\right) }\right\vert ^{\frac{st}{2}}dt+\int_{0}^{1}\left\vert
t\right\vert \left\vert \frac{f^{\prime }\left( b\right) }{f^{\prime }\left(
a\right) }\right\vert ^{\frac{st}{2}}dt\right]  \notag
\end{eqnarray}%
for all $t\in \left[ 0,1\right] .$ Using the well known inequality $mn\leq
\mu m^{\frac{1}{\mu }}+\eta n^{\frac{1}{\eta }},$ on the right side of (\ref%
{k}), we get%
\begin{eqnarray*}
&&\left\vert \frac{f\left( a\right) +f\left( b\right) }{2}-\frac{1}{b-a}%
\int_{a}^{b}f\left( x\right) dx\right\vert \\
&\leq &\frac{b-a}{4}\left\vert f^{\prime }\left( a\right) f^{\prime }\left(
b\right) \right\vert ^{\frac{s}{2}}\left\{ \left\vert -t\right\vert
\left\vert \frac{f^{\prime }\left( a\right) }{f^{\prime }\left( b\right) }%
\right\vert ^{\frac{st}{2}}dt+\int_{0}^{1}\left\vert t\right\vert \left\vert 
\frac{f^{\prime }\left( b\right) }{f^{\prime }\left( a\right) }\right\vert ^{%
\frac{st}{2}}dt\right\} \\
&\leq &\frac{b-a}{4}\left\vert f^{\prime }\left( a\right) f^{\prime }\left(
b\right) \right\vert ^{\frac{s}{2}}\left\{ \mu _{1}\int_{0}^{1}\left\vert
-t\right\vert ^{\frac{1}{\mu _{1}}}dt+\eta _{1}\int_{0}^{1}\left\vert \frac{%
f^{\prime }\left( a\right) }{f^{\prime }\left( b\right) }\right\vert ^{\frac{%
s}{2\eta _{1}}t}dt\right. \\
&&+\left. \mu _{2}\int_{0}^{1}\left\vert t\right\vert ^{\frac{1}{\mu _{2}}%
}dt+\eta _{2}\int_{0}^{1}\left\vert \frac{f^{\prime }\left( b\right) }{%
f^{\prime }\left( a\right) }\right\vert ^{\frac{s}{2\eta _{2}}t}dt\right\} \\
&=&\frac{b-a}{4}\left\vert f^{\prime }\left( a\right) f^{\prime }\left(
b\right) \right\vert ^{\frac{s}{2}}\left[ \frac{\mu _{1}^{2}}{1+\mu _{1}}%
+\eta _{1}\int_{0}^{1}\left\vert \frac{f^{\prime }\left( a\right) }{%
f^{\prime }\left( b\right) }\right\vert ^{\frac{s}{2\eta _{1}}t}dt+\frac{\mu
_{2}^{2}}{1+\mu _{2}}+\eta _{2}\int_{0}^{1}\left\vert \frac{f^{\prime
}\left( b\right) }{f^{\prime }\left( a\right) }\right\vert ^{\frac{s}{2\eta
_{2}}t}dt\right] \\
&=&\frac{b-a}{4}\left\vert f^{\prime }\left( a\right) f^{\prime }\left(
b\right) \right\vert ^{\frac{s}{2}}\left\{ \frac{\mu _{1}^{2}}{1+\mu _{1}}%
+\eta _{1}g_{2}\left( \alpha \left( \frac{s}{2\eta _{1}},\frac{s}{2\eta _{1}}%
\right) \right) \right. \\
&&+\left. \frac{\mu _{2}^{2}}{1+\mu _{2}}+\eta _{2}g_{2}\left( \alpha \left( 
\frac{s}{2\eta _{2}},\frac{s}{2\eta _{2}}\right) \right) \right\} \\
&=&\frac{b-a}{4}\left\vert f^{\prime }\left( a\right) f^{\prime }\left(
b\right) \right\vert ^{\frac{s}{2}}\left\{ \frac{\left( 1+\mu _{2}\right)
\mu _{1}^{2}+\left( 1+\mu _{1}\right) \mu _{2}^{2}}{\left( 1+\mu _{1}\right)
\left( 1+\mu _{2}\right) }\right. \\
&&+\left. \eta _{1}g_{2}\left( \alpha \left( \frac{s}{2\eta _{1}},\frac{s}{%
2\eta _{1}}\right) \right) +\eta _{2}g_{2}\left( \alpha \left( \frac{s}{%
2\eta _{2}},\frac{s}{2\eta _{2}}\right) \right) \right\}
\end{eqnarray*}%
and we get, in here, if $\left\vert f^{\prime }\left( a\right) \right\vert
=\left\vert f^{\prime }\left( b\right) \right\vert =1,$ we get%
\begin{equation*}
\left\vert \frac{f\left( a\right) +f\left( b\right) }{2}-\frac{1}{b-a}%
\int_{a}^{b}f\left( x\right) dx\right\vert \leq \frac{b-a}{4}\left[ \frac{%
\left( 1+\mu _{2}\right) \mu _{1}^{2}+\left( 1+\mu _{1}\right) \mu _{2}^{2}}{%
\left( 1+\mu _{1}\right) \left( 1+\mu _{2}\right) }+\eta _{1}+\eta _{2}%
\right]
\end{equation*}%
which the proof is completed.
\end{proof}

\section{Applications to special means for positive numbers}

Let%
\begin{eqnarray*}
A\left( a,b\right) &=&\frac{a+b}{2},\text{ }L\left( a,b\right) =\frac{b-a}{%
\ln b-\ln a}\text{ \ \ \ \ }\left( a\neq b\right) , \\
L_{p}\left( a,b\right) &=&\left( \frac{b^{p+1}-a^{p+1}}{\left( p+1\right)
\left( b-a\right) }\right) ^{1/p},\text{ }a\neq b,\text{ }p\in 
%TCIMACRO{\U{211d} }%
%BeginExpansion
\mathbb{R}
%EndExpansion
,\text{ }p\neq -1,0
\end{eqnarray*}%
be the arithmetic, logarithmic, generalized logarithmic means for $a,b>0$
respectively.

In the following propositions, $\alpha \left( u,v\right) =\frac{\left\vert
f^{\prime }\left( a\right) \right\vert ^{u}}{\left\vert f^{\prime }\left(
b\right) \right\vert ^{v}}=\frac{\left\vert a^{s-1}\right\vert ^{u}}{%
\left\vert b^{s-1}\right\vert ^{v}}$. \ \ 

\begin{proposition}
Let $0<a<b\leq 1,$ with $a\neq b,$ and $0<s<1.$ Then, we have%
\begin{eqnarray*}
&&\frac{1}{s}\left\vert A\left( a^{s},b^{s}\right) -L_{s}\left( a,b\right)
^{s}\right\vert \\
&\leq &\frac{\left( b-a\right) }{4}\left( ab\right) ^{\frac{s}{2}\left(
s-1\right) }\left\{ \frac{\left\vert \frac{a}{b}\right\vert ^{\left(
s-1\right) \frac{s}{2}}\ln \left\vert \frac{a}{b}\right\vert ^{\left(
s-1\right) \frac{s}{2}}-\left\vert \frac{a}{b}\right\vert ^{\left(
s-1\right) \frac{s}{2}}+1}{\left( \ln \left\vert \frac{a}{b}\right\vert
^{\left( s-1\right) \frac{s}{2}}\right) ^{2}}\right. \\
&&+\left. \frac{\left\vert \frac{a}{b}\right\vert ^{-\left( s-1\right) \frac{%
s}{2}}\ln \left\vert \frac{a}{b}\right\vert ^{-\left( s-1\right) \frac{s}{2}%
}-\left\vert \frac{a}{b}\right\vert ^{-\left( s-1\right) \frac{s}{2}}+1}{%
\left( \ln \left\vert \frac{a}{b}\right\vert ^{-\left( s-1\right) \frac{s}{2}%
}\right) ^{2}}\right\}
\end{eqnarray*}
\end{proposition}

\begin{proof}
Let $f\left( x\right) =\frac{x^{s}}{s},$ $x\in \left( 0,1\right] ,$ $0<s<1,$
then $\left\vert f^{\prime }\left( x\right) \right\vert =x^{s-1},$ $x\in
\left( 0,1\right] $ is a $s$-geometrically convex mapping. The assertion
follows from Theorem \ref{t1} applied to $s$-geometrically convex mapping $%
\left\vert f^{\prime }\left( x\right) \right\vert =x^{s-1},$ $x\in \left( 0,1%
\right] .$
\end{proof}

\begin{example}
Let $f\left( x\right) =\frac{x^{s}}{s},$ $x\in \left( 0,1\right] ,$ $0<s<1,$
then $\left\vert f^{\prime }\left( x\right) \right\vert =x^{s-1},$ $x\in
\left( 0,1\right] $ is a $s$-geometrically convex mapping. If we apply in
Theorem \ref{t1}, for $s=0.5,a=0.89,b=0.9,$ we get%
\begin{eqnarray*}
&&\frac{1}{s}\left\vert \frac{a^{s}+b^{s}}{2}-\left( \frac{b^{s+1}-a^{s+1}}{%
\left( s+1\right) \left( b-a\right) }\right) \right\vert \\
&=&4.\,\allowbreak 921\,067\,116\times 10^{-6}\allowbreak \\
&\leq &\frac{\left( b-a\right) }{4}\left\vert ab\right\vert ^{\frac{s}{2}%
\left( s-1\right) } \\
&&\times \left( \frac{\left\vert \frac{a}{b}\right\vert ^{\left( s-1\right) 
\frac{s}{2}}\ln \left\vert \frac{a}{b}\right\vert ^{\left( s-1\right) \frac{s%
}{2}}-\left\vert \frac{a}{b}\right\vert ^{\left( s-1\right) \frac{s}{2}}+1}{%
\left( \ln \left\vert \frac{a}{b}\right\vert ^{\left( s-1\right) \frac{s}{2}%
}\right) ^{2}}+\frac{\left\vert \frac{a}{b}\right\vert ^{-\left( s-1\right) 
\frac{s}{2}}\ln \left\vert \frac{a}{b}\right\vert ^{-\left( s-1\right) \frac{%
s}{2}}-\left\vert \frac{a}{b}\right\vert ^{-\left( s-1\right) \frac{s}{2}}+1%
}{\left( \ln \left\vert \frac{a}{b}\right\vert ^{-\left( s-1\right) \frac{s}{%
2}}\right) ^{2}}\right) \\
&=&2.\,\allowbreak 570\,313\,847\times 10^{-3}
\end{eqnarray*}%
And similarly, if we apply for $s=0.2,a=0.15,b=0.6,$ we obtain%
\begin{equation*}
9.\,\allowbreak 780\,804\,473\times 10^{-2}\allowbreak \leq
0.136\,819\,309\,\allowbreak 576\,863\,680\,\allowbreak 170\,486
\end{equation*}%
for $s=0.75,a=0.45,b=0.86,$ we obtain 
\begin{equation*}
6.\,\allowbreak 115\,413\,651\times 10^{-2}\allowbreak \leq
0.112\,144\,032\,\allowbreak 368\,736\,206\,\allowbreak 184\,243
\end{equation*}%
etc.
\end{example}

\begin{proposition}
Let $0<a<b\leq 1,$with $a\neq b,$ and $0<s<1,$ and $p,q>1.$ Then, we have%
\begin{eqnarray*}
&&\frac{1}{s}\left\vert A\left( a^{s},b^{s}\right) -L_{s}\left( a,b\right)
^{s}\right\vert \\
&\leq &\frac{\left( b-a\right) }{4\left( p+1\right) ^{\frac{1}{p}}}%
\left\vert ab\right\vert ^{\frac{s}{2}\left( s-1\right) }\left( b^{s\left(
1-s\right) /2}+a^{s\left( 1-s\right) /2}\right) \left[ L\left( a^{\left(
s-1\right) \frac{sq}{2}},b^{\left( s-1\right) \frac{sq}{2}}\right) \right] ^{%
\frac{1}{q}}.
\end{eqnarray*}
\end{proposition}

\begin{proof}
\ The assertion follows from Theorem \ref{t2} applied to $s$-geometrically
convex mapping $\left\vert f^{\prime }\left( x\right) \right\vert =\frac{%
x^{s}}{s},$ $x\in \left( 0,1\right] .$
\end{proof}

\begin{proposition}
Let $0<a<b\leq 1,$ with $a\neq b,$ and $0<s<1,$ and $q\geq 1.$ Then, we have%
\begin{eqnarray*}
&&\frac{1}{s}\left\vert A\left( a^{s},b^{s}\right) -L_{s}\left( a,b\right)
^{s}\right\vert \\
&\leq &\frac{b-a}{4}\left( \frac{1}{2}\right) ^{1-\frac{1}{q}}\left\{
\left\vert \frac{a}{b}\right\vert ^{\frac{s}{2}\left( s-1\right) }\left[ 
\frac{\left\vert \frac{a}{b}\right\vert ^{\left( s-1\right) \frac{sq}{2}}\ln
\left\vert \frac{a}{b}\right\vert ^{\left( s-1\right) \frac{sq}{2}%
}-\left\vert \frac{a}{b}\right\vert ^{\left( s-1\right) \frac{sq}{2}}+1}{%
\left( \ln \left\vert \frac{a}{b}\right\vert ^{\left( s-1\right) \frac{sq}{2}%
}\right) ^{2}}\right] ^{\frac{1}{q}}\right. \\
&&+\left. \left\vert \frac{b}{a}\right\vert ^{\frac{s}{2}\left( s-1\right) }%
\left[ \frac{\left\vert \frac{a}{b}\right\vert ^{-\left( s-1\right) \frac{sq%
}{2}}\ln \left\vert \frac{a}{b}\right\vert ^{-\left( s-1\right) \frac{sq}{2}%
}+\left\vert \frac{a}{b}\right\vert ^{-\left( s-1\right) \frac{sq}{2}}-1}{%
\left( \ln \left\vert \frac{a}{b}\right\vert ^{-\left( s-1\right) \frac{sq}{2%
}}\right) ^{2}}\right] ^{\frac{1}{q}}\right\} .
\end{eqnarray*}
\end{proposition}

\begin{proof}
The assertion follows from Theorem \ref{t3} applied to $s$-geometrically
convex mapping $\left\vert f^{\prime }\left( x\right) \right\vert =\frac{%
x^{s}}{s},$ $x\in \left( 0,1\right] .$
\end{proof}

\begin{proposition}
Let $0<a<b\leq 1,$ with $a\neq b,$ and $0<s<1,$ and $q\geq 1.$ Then, we have%
\begin{eqnarray*}
&&\frac{1}{s}\left\vert A\left( a^{s},b^{s}\right) -L_{s}\left( a,b\right)
^{s}\right\vert \\
&\leq &\frac{b-a}{4}\left\vert ab\right\vert ^{\frac{s}{2}\left( s-1\right) }
\\
&&\times \left[ \frac{\left( 1+\mu _{2}\right) \mu _{1}^{2}+\left( 1+\mu
_{1}\right) \mu _{2}^{2}}{\left( 1+\mu _{1}\right) \left( 1+\mu _{2}\right) }%
+\eta _{1}\left[ \frac{\left\vert \frac{a}{b}\right\vert ^{\left( s-1\right) 
\frac{s}{2\eta _{1}}}-1}{\ln \left\vert \frac{a}{b}\right\vert ^{\left(
s-1\right) \frac{s}{2\eta _{1}}}}\right] +\eta _{2}\left[ \frac{\left\vert 
\frac{a}{b}\right\vert ^{\left( s-1\right) \frac{s}{2\eta _{2}}}-1}{\ln
\left\vert \frac{a}{b}\right\vert ^{\left( s-1\right) \frac{s}{2\eta _{2}}}}%
\right] \right] .
\end{eqnarray*}
\end{proposition}

\begin{proof}
\ The assertion follows from Theorem \ref{t4} applied to $s$-geometrically
convex mapping $\left\vert f^{\prime }\left( x\right) \right\vert =\frac{%
x^{s}}{s},$ $x\in \left( 0,1\right] .$
\end{proof}


\begin{thebibliography}{9}
\bibitem{hud} H. Hudzik and L. Maligranda: Some remarks on $s$-convex
functions,\ Aequationes Math., Vol. 48 (1994), 100--111.

\bibitem{lmmm} M. Alomari, M. Darus, U. S. K\i rmac\i : Refinements of
Hadamard-type inequalities for quasi-convex functions with applications to
trapezoidal formula and to special means, Comp. and Math. with Appl., Vol.59
(2010), 225-232.

\bibitem{zh} T.-Y. Zhang, A.-P. Ji and F. Qi: On Integral \i nequalities of
Hermite-Hadamard Type for s-Geometrically Convex Functions. Abstract and
Applied Analysis. doi:10.1155/2012/560586.

\bibitem{dr} S.S. Dragomir, R.P. Agarwal: Two inequalities for
differentiable mappings and applications to special means of real numbers
and to trapezoidal formula. Appl Math Lett, Vol. 11 No:5, (1998) 91--95.

\bibitem{hdm} J. Hadamard: \'{E}tude sur les propri\'{e}t\'{e}s des
fonctions enti\`{e}res et en particulier d'une fonction consider\'{e}e par
Riemann, J. Math Pures Appl., 58, (1893) 171--215.

\bibitem{tnc2} M. Tun\c{c}: On some new inequalities for convex fonctions,
Turk. J. Math. 36 (2012), 245-251.

\bibitem{mit} D. S. Mitrinovi\'{c}, J. Pe\v{c}ari\'{c} and A. M. Fink:\
Classical and new inequalities in analysis, KluwerAcademic, Dordrecht, 1993.

\bibitem{dr2} S. S. Dragomir and C. E. M. Pearce: Selected topics on
Hermite-Hadamard inequalities and applications, RGMIA monographs, Victoria
University, 2000. [Online:\
http://www.staff.vu.edu.au/RGMIA/monographs/hermite-hadamard.html].

\bibitem{pec} J. E. Pe\v{c}ari\'{c}, F. Proschan and Y. L. Tong: Convex
Functions, Partial Orderings, and Statistical Applications, Academic Press
Inc., 1992.
\end{thebibliography}
\end{document}